\documentclass[letterpaper, 10 pt, conference]{ieeeconf}

\IEEEoverridecommandlockouts % This command is only
\let\labelindent\relax
\usepackage{url}
\usepackage{float}
\usepackage{enumitem}
\usepackage{graphicx}
\usepackage{amsthm,amssymb,amsfonts,thmtools,thm-restate}
\usepackage{grffile}
\usepackage{bigints}
\usepackage{tikz}
\usetikzlibrary{shapes.geometric, arrows, positioning, calc}
\usepackage{mathdots}
\usepackage{cite,cleveref}
\usepackage{balance}
\usepackage{afterpage}
\newsavebox{\shortpagebox}

\makeatletter
\newcommand{\shortpage}[1]% #1= \twocolumn text to wrap into \onecolumn page
{\par
\setbox\shortpagebox=\vbox{\strut #1\par}%
\afterpage{\onecolumn
\begin{multicols}{2}
\unvbox\AP@partial
\end{multicols}}%
\unvbox\shortpagebox
\par}
\makeatother

\newcommand{\mcl}[1]{\mathcal{ #1}}
\newcommand{\mbf}[1]{\mathbf{ #1}}
\newcommand{\norm}[1]{\left\Vert #1\right\Vert}

\newcommand{\ip}[2]{\left\langle{#1},{#2}\right\rangle}

\newcommand{\bmat}[1]{\begin{bmatrix} #1\end{bmatrix}}

\newcommand{\mat}[1]{\begin{matrix}#1\end{matrix}}

\newcommand{\R}{\mathbb{R}}

\newcommand{\N}{\mathbb{N}}

\newtheorem{definition}{Definition}

\crefname{EqnBlock}{Figure}{Figures}

\newcommand{\threepi}[1]{\mcl{P}_{\{#1\}}}

\title{Representation of linear PDEs with spatial integral terms as Partial Integral Equations}

\author{Sachin Shivakumar$^{1}$ \and Amritam Das$^{2}$ \and Matthew M. Peet$^{1}$
\thanks{$^{1}$ Sachin Shivakumar\{sshivak8@asu.edu\} and Matthew M. Peet\{mpeet@asu.edu\} are with School for Engineering of Matter, Transport and Energy, Arizona State University, USA}%
\thanks{$^{2}$ Amritam Das\{am.das@tue.nl\} is with Eindhoven University of Technology, Netherlands}%
}

\allowdisplaybreaks
\begin{document}
\maketitle

\begin{abstract}
In this paper, we present the Partial Integral Equation (PIE) representation of linear Partial Differential Equations (PDEs) in one spatial dimension, where the PDE has spatial integral terms appearing in the dynamics and the boundary conditions. The PIE representation is obtained by performing a change of variable where every PDE state is replaced by its highest, well-defined derivative using the Fundamental Theorem of Calculus to obtain a new equation (a PIE). We show that this conversion from PDE representation to PIE representation can be written in terms of explicit maps from the PDE parameters to PIE parameters. Lastly, we present numerical examples to demonstrate the application of the PIE representation by performing stability analysis of PDEs via convex optimization methods.
\end{abstract}

%%%%%%%%%%%%%%%%%%%%%%%%%%%%%%%%%%%%%%%%%%%%%%%%%%%%%%%%%%%%%%%%%%%%%%%%%%%%%%%%
\section{Introduction}\label{sec:Introduction}

Certain Partial Differential Equations (PDEs) with spatial integral terms can be represented as Partial Integral Equations (PIEs), which then can be subjected to analysis/control techniques developed for PIEs. While a formal definition of the class of PDEs is presented in Section \ref{sec:PDE}, here we briefly describe the form of these spatial integral terms. Given a $N$-times spatially differentiable state $\mbf x$ in time $t$ and space $s\in [a,b]$, we refer to terms of the form $\int_{\Omega} F(s)\partial_s^k \mbf x(t,\theta) d\theta$ as spatial integral terms, where $0\le k\le N$. The integral can be definite ($\Omega = [a,b]$) or indefinite ($\Omega = [a,s]$ or $[s,b]$). When terms of this form appear in the dynamics or the boundary conditions of a PDE, we say the PDE has spatial integral terms.

PDEs, that have spatial integral terms, appear in various models of population dynamics \cite{population_dynamics}, linear thermoelasticity \cite{day2013heat}, chemical reaction and transport \cite{chemical}, etc.). Alternatively, PDE models involving physical point sensor measurements also have spatial integral terms because, in practice, point measurements are spatial averages of the measured quantity. The closed-loop dynamics of PDEs (without integral terms) with a backstepping controller or observer \cite{krstic2008boundary} typically have spatial integral terms.

\textit{Example 1:} To demonstrate a simple example of a PDE with integral terms in the boundary condition, we consider the McKendrick PDE used to model population dynamics; See \cite{population_dynamics}). The PDE model has the form
\begin{align*}
\dot{x}(t,s) &= -\partial_s x(t,s) + f(s)x(t,s),\quad s\in[0,1],\\
x(t,0) &= \int_0^1 h(s)x(t,s)ds,
\end{align*}
where $x$ is the density of the population of age $s$ at any given time $t$. The boundary condition can be considered an approximation of the number of newborns at time $t$ and is usually modeled as an integral boundary condition (a weighted average of the population density at different ages).

\textit{Example 2:} Next, we present an example where a simple PDE without integral terms, coupled with a boundary sensed observer, becomes a PDE with spatial integral terms. Let us take the reaction-diffusion PDE model with a boundary sensed observer. The equations can be written as:
\begin{align*}
&\dot{x}(t,s) = \lambda x(t,s) + \partial_s^2 x(t,s), \quad s\in [0,1],\\
&\dot{\hat x}(t,s) = \lambda \hat x(t,s) + \partial_s^2 \hat x(t,s) + l(s) \left(\partial_s x(t,1) - \partial_s \hat x(t,1)\right),\\
&\partial_s x(t,0) = 0, \quad x(t,1) = 0,\\
&\partial_s \hat x(t,0) = 0, \quad \hat x(t,1) = 0,
\end{align*}
where $x$ is the distributed PDE state, $\hat x$ is the observer state, and $l$ is the observer gain. A physical implementation of a boundary sensor leads to the appearance of integral terms because point measurements such as $y(t) = \partial_s x(t,1)$ (the sensor measurement that drives the observer dynamics) are not, typically, exact due to the limitations of the sensing mechanism. A more accurate representation of such an observer is given by
\begin{align*}
&\dot{\hat x}(t,s) = \lambda \hat x(t,s) + \partial_s^2 \hat x(t,s) \\
&\quad + \int_0^1l(s)w(\theta)\left(\partial_s x(t,\theta) - \partial_s \hat x(t,\theta)\right)d\theta.
\end{align*}
where $w$ is a Gaussian function centered at $s=1$.

Various methods have been developed for analyzing such PDEs, such as eigenvalue-based analysis methods presented in \cite{lasiecka_triggiani_2000}. However, the presence of integral terms can cause errors and instabilities; a closed form of the eigenfunctions for PDEs with integral terms cannot be found and must be approximated numerically --- giving rise to approximation errors and numerical instabilities. Some mathematical analysis methods were presented, specifically for PDEs with integral terms, in \cite{appell2000partial, gil2015stability} whereas \cite{yang2017stability} used computational tools like LMIs for analysis. These methods did not have approximation errors; however, they did not consider PDEs with integral terms at the boundary.

Our approach for analysis/control of PDEs involves the use of PIE representation. Since PIE representation of linear PDEs was introduced in \cite{peet_2018CDC}, various analysis and control methods for PIEs have been developed (see \cite{shivakumar_SICON2022}). These analysis/control methods are based on convex-optimization problems that can be solved using Linear Matrix Inequalities (LMIs) and do not involve any approximation. If we can find a PIE representation for PDEs with integral terms, then we can use analysis/control methods for the PIE to solve the analysis/control problem of such PDEs.

Thus, this paper aims to find the conditions under which a PDE with integral terms has a PIE representation and then find the PIE representation if it exists. To achieve these goals, we first present a standard parametric representation for the PDEs with integral terms that allows integral terms in the dynamics and the boundary. Then, based on this standard representation, we find a sufficient condition for the existence of a PIE representation and then show that the two representations are equivalent by finding a bijective map from the solution of the PDE to the solution of the corresponding PIE.

To summarize the main contribution of this paper, we provide a test for the existence of a PIE representation for a given PDE with integral terms. Then, we also present explicit maps from the parameters of a PDE with integral terms to the parameters of the PIE where the PIE, so obtained, will have stability properties identical to that of the PDE. Then, we present the stability test for PDE with integral terms as an operator-valued optimization problem. While the equivalence of properties of a PDE and the corresponding PIE can be easily extended to input-output properties (see \cite{shivakumar2019generalized}), that topic is beyond the scope of this paper and will not be discussed here.

The paper is organized as follows: First, to introduce the ``target'' representation (a PIE), we define the algebra of Partial integral (PI) operators in \Cref{subsec:PI} and then define the class of PIE models, along with a definition of a solution in \Cref{subsec:PIE}. We introduce the class of PDE systems in \Cref{sec:PDE} for which a PIE representation will be found. Next, we propose an admissibility condition that guarantees the existence of an equivalent PIE representation for a given PDE and present the formulae for constructing the PIE representation of the PDE (See \Cref{sec:pde2pie}). Finally, in \Cref{sec:numerical}, we use specific PDE models to test stability.

\section{Notation}\label{sec:notation}
We use the notation $0_{m\times n}$ to represent the zero matrix of dimension $m\times n$ and $0_n:=0_{n\times n}$. Similarly, $I_n$ is the identity matrix of dimension $n\times n$. When the dimensions are clear from the context, we use $0$ and $I$ for the zero and identity matrix. $\R_+$ is the set of non-negative real numbers.

$L_2^n[a,b]$ is the Hilbert space of $n$-dimensional vector-valued Lebesgue square-integrable functions defined on the interval $[a,b]$ and is equipped with the standard inner product. $L_{\infty}[a,b]^{m,n}$ is the Banach space of $m\times n$-dimensional essentially bounded measurable matrix-valued functions on $[a,b]$ equipped with the essential supremum singular value norm.

The bold font, $\mbf x$ or $\mbf x(t)$, typically implies that $\mbf x$ or $\mbf x(t)$ is a scalar or vector-valued function (e.g. $\mbf x(t)\in L_2^n[a,b]$) whereas the normal font $x(t)$ implies that $x(t)$ is a finite-dimensional vector. For a suitably differentiable function, $\mbf x$ of spatial variable $s$, we use $\partial_s^j \mbf x$ to denote the $j$-th order partial derivative $\frac{\partial^j\mbf x}{\partial s^j}$. For a suitably differentiable function of time and possibly space, we denote $\dot{\mbf x}(t)=\frac{\partial }{\partial t}\mbf x(t)$. We use $W_{k}^n$ to denote the Sobolev spaces \begin{align}\label{eq:sobolev}
W_{k}^n[a,b] := \{\mbf u\in L_2^n[a,b]\mid \partial^l_s \mbf u\in L_2^n[a,b] ~\text{for all} ~l\le k\}.
\end{align}
with inner product $\ip{\mbf u}{\mbf v}_{H_{k}}=\sum\nolimits_{i=0}^k \ip{\partial_s^i\mbf u}{\partial_s^i \mbf v}_{L_2^n}$. For brevity, we omit the domain $[a,b]$ and write $L_2^n$ or $W_{k}^n$ when clear from the context.

% Finally, for normed spaced $A,B$, $\mcl L(A,B)$ denotes the space of bounded linear operators from $A$ to $B$ equipped with induced operator norm. $\mcl L(A) := \mcl L(A,A)$
\section{PI Operators and PIE representation}\label{subsec:PI}
PIEs are defined by bounded, linear maps from $L_2^n \rightarrow L_2^n$ called Partial Integral (PI) operators, which are parameterized by separable functions. To quickly recall, a separable function is defined as follows.

\begin{definition}[Separable Function]\label{def:separable}
We say a function $R: [a,b]^{2} \rightarrow \R^{p \times q}$, is separable if there exist $r \in \N$ and functions $F\in L_{\infty}^{r\times p}[a,b]$ and $G\in L_{\infty}^{r\times q}[a,b]$ such that $R(s,\theta)=F(s)^T G(\theta)$.
\end{definition}

Using the above definition, we can define $3$-PI operators (with three parameters) as follow.
\begin{definition}[3-PI operators, $\Pi_3$]\label{def:3PI}
Given $R_0\in L_{\infty}^{p \times q}[a,b]$ and separable functions $R_1, R_2: [a,b]^{2} \rightarrow \R^{p \times q}$, we define the operator $\mcl{P}_{\{R_i\}}$ for $\mbf v \in L_2$ as
\begin{align}\label{eq:3pi}
\left(\mcl{P}_{\{R_i\}}\mbf v\right)(s)&:= R_0(s) \mbf v(s) +\int_{a}^s R_1(s,\theta)\mbf v(\theta)d \theta\notag\\
&\quad +\int_s^bR_2(s,\theta)\mbf v(\theta)d \theta.
\end{align}
Furthermore, we say an operator, $\mcl P$, is 3-PI of dimension $p \times q$, denoted $\mcl P \in [\Pi_3]_{p,q}\subset \mcl L(L_2^q,L_2^p)$, if there exist functions $R_0\in L_{\infty}^{p\times q}$ and separable functions $R_1,R_2$ such that $\mcl P=\mcl{P}_{\{R_i\}}$.
\end{definition}
Proving that the set of 3-PI operators form a $^*$-algebra is beyond the scope of this paper; however, the proof can be found in \cite{shivakumar_SICON2022}. Additionally, we define $[\Pi_3]_{n,n}^+$ as the set of positive definite 3-PI operators where the positivity is defined with respect to $L_2$-inner product, i.e, $[\Pi_3]_{n,n}^+ := \{ \mcl P \in [\Pi_3]_{n,n}\mid \ip{x}{\mcl P x}_{L_2} > 0 ~\forall x\in L_2~\text{and}~x\ne 0\}$.

%\textbf{Note:} An important property of the PI algebras is that algebraic operations on the operator can be represented using algebraic operations on the parameters which represent that operator.

\subsection{Partial Integral Equations}\label{subsec:PIE}
A PIE is an extension of the state-space representation of ODEs (vector-valued first-order differential equations on $\R^n$) to spatially-distributed states on $L_2$. A PIE model is parameterized by 3-PI operators as
\begin{align}\label{eq:PIE_full}
\mcl{T}\dot{\mbf{x}}_f(t)&=\mcl{A}\mbf x_f(t), \mbf{x}_f(0) = \mbf{x}_f^0\in L_2^{m,n}[a,b],
\end{align}
where $\mcl T\in [\Pi_3]_{n,n}$ and $\mcl A\in [\Pi_3]_{n,n}$ are 3-PI operators.

Unlike PDE models, a PIE does not allow for spatial derivatives -- only a first-order derivative with respect to time. In particular, the state of the PIE model, $\mbf x_f \in L_2[a,b]$, is not differentiable; consequently, no boundary conditions are possible in a PIE model. Finally, given a PIE, we require differentiability of the solution with respect to the $\mcl{T}$-norm, which is defined as
\begin{align}\label{eq:Tnorm}
\norm{\mbf{x}_f}_{\mcl{T}} := \norm{\mcl{T}\mbf{x}_f}_{L_2}, \qquad \text{for}~\mbf{x}_f\in L_2.
\end{align}
We require the map $\mcl T$ in a PIE model (if the PIE model corresponds to a PDE) to be bijective, and hence one can easily prove that $\norm{\cdot}_{\mcl T}$ is a norm (see \cite{shivakumar_SICON2022}). The solution, if it exists, of a PIE must satisfy the following requirements.

\begin{definition}[Definition of solution of a PIE]\label{def:piesolution}
For given inputs $\mbf x_f^0\in L_2^{n}$, we say that $\{\mbf x_f\}$ satisfies the PIE defined by $\{\mcl{T},$ $\mcl{A}\}$ with initial condition $\mbf x_f^0$ if: a) $\mbf x_f(t)\in L_2^{n}[a,b]$ for all $t\ge 0$; b) $\mbf x_f$ is Frech\'et differentiable with respect to the $\mcl{T}$-norm almost everywhere on $\R_+$; c) $\mbf x_f(0) = \mbf x_f^0$; and d) Eq.~\eqref{eq:PIE_full} is satisfied for almost all $t\in\R_+$.
\end{definition}

\section{A parametric form of PDEs with integral terms}\label{sec:PDE}

Now that we have defined the ``target'' representation, we need to define the class of PDEs that will be converted to this target representation. To define the PDE, we first categorize the parameters of a PDE into three groups based on the type of constraint they appear in -- namely the continuity constraints, the in-domain dynamics, and the boundary conditions. The continuity constraints specify the existence of partial derivatives and boundary values for each state as required by the in-domain dynamics and boundary conditions. The boundary conditions are represented as a real-valued algebraic constraint that maps the distributed state to a vector of boundary values. The in-domain dynamics (or generating equation) specify the time derivative of the state at every point in the interior of the domain and allow for both integral and derivative operators in the spatial variable $s$. The following subsections highlight the parameters required to define the above three types of constraints in a PDE.

\subsubsection{The continuity constraint} Given a PDE state, $\mbf x(t,\cdot)$, the continuity constraint can be uniquely defined by the parameter $n=\{n_0,n_1,n_2\}$, which partitions and orders the PDE states $\mbf x$ by increasing differentiability as follows.
\[
\mbf x(t,\cdot)=\bmat{\mbf x_0(t,\cdot)\\ \mbf x_1(t,\cdot)\\\mbf x_2(t,\cdot)} \in \bmat{W_0^{n_0}\\W_1^{n_1} \\ W_2^{n_2}} =: W^n.
\]
Given such an $n\in \N^3$, we can identify all well-defined partial derivatives of $\mbf x$.

\noindent \textbf{Notation:} For convenience, we define the vector of all continuous partial derivatives of the PDE state $\mbf x$ as permitted by the continuity constraint as $\mbf x_c$, the vector of all partial derivatives as $\mbf x_D$ and the list of all possible boundary values of $\mbf x$ as $x_b$, i.e.,
\begin{flalign}\label{eq:odepde_bc}
%%%%%%%%%%%%%%%%%%%%%%%%%%%%%%%%%%%
\mat{\mbf x_c(t,\cdot)=\bmat{\mbf x_1(t,\cdot)\\\mbf x_2(t,\cdot)\\\partial_s \mbf x_2(t,\cdot)}\\\\\quad x_b(t)=\bmat{\mbf x_c(t,a)\\\mbf x_c(t,b)}},\quad \mbf x_D(t,\cdot) = \bmat{\mbf x_0(t,\cdot)\\\mbf x_1(t,\cdot)\\\mbf x_2(t,\cdot)\\\partial_s\mbf x_1(t,\cdot)\\\partial_s\mbf x_2(t,\cdot)\\\partial_s^2 \mbf x_2(t,\cdot)}.
\end{flalign}
Additionally, we define some notation given the continuity constraint parameter $n=\{n_0,n_1,n_2\}$ that will repeatedly appear in the subsequent sections. Let $n_{\mbf x}:=\sum_{i=0}^2 n_i$ be the number of states in $\mbf x$ and $n_{S}=\sum_{i=0}^2 i\cdot n_{i}$.
\subsubsection{Boundary Conditions}
Given an $n=\{n_0,n_1,n_2\}$, we now parameterize a generalized class of boundary conditions consisting of a combination of boundary values and integrals of the PDE state. Specifically, the boundary conditions are parameterized by square integrable functions $B_{I}(s) \in \R^{n_{BC} \times (n_{\mbf x}+n_{S})}$ and $B\in \R^{n_{BC}\times 2n_S}$ as
\begin{flalign}\label{eq:odepde_b}
0 &= \int_{a}^{b} B_{I}(s)\mbf x_D(t,s)ds -Bx_b(t)
\end{flalign}
where $n_{BC}$ is the number of specified boundary conditions. For reasons of well-posedness, as discussed in \Cref{sec:pde2pie}, we typically require $n_{BC}=n_{S}$.

These boundary conditions and the continuity constraints collectively define the domain of the infinitesimal generator -- which specifies a set of acceptable solutions $\mbf x(t)\in X$ for the PDE -- as
\begin{align}\label{eq:odepde_general_domain}
&X = \notag\\
&\left\lbrace \mbf{x}\in W^{n}[a,b]: B\bmat{\mbf x_c(t,a)\\\mbf x_c(t,b)}=\int\limits_{a}^{b} B_{I}(s)\mbf{x}_D(t,s)ds\right\rbrace.
\end{align}

\noindent\textbf{Notation:} For convenience, we collect all the parameters which define the boundary-valued constraint in Eq.~\eqref{eq:odepde_b} and introduce the shorthand notation $\mbf G_{\mathrm{b}}$ which represents the labeled tuple of system parameters as
\begin{align}
&\mbf G_{\mathrm{b}} = \left\{B,~B_{I}\right\}.\label{eq:BC-parms}
\end{align}
When this shorthand notation is used to denote a given set of system parameters, it is presumed that all parameters have appropriate dimensions.

\subsubsection{Dynamics of the PDE}
Finally, we may now define the dynamics of the PDE which is parameterized by the functions $A_{0}(s),$ $A_{1}(s,\theta),$ and $A_{2}(s,\theta)$ $\in \R^{n_{\mbf x} (n_{\mbf x}+n_{S})}$ as
\begin{flalign}\label{eq:general_pde_}
\dot{\mbf{x}}(t,s) &= A_0(s)\mbf x_D(t,s) +\int\limits_{a}^{s}A_1(s,\theta)\mbf x_D(t,\theta)d\theta\notag\\
&\quad+\int\limits_{s}^{b}A_2(s,\theta)\mbf x_D(t,\theta)d\theta,
\end{flalign}
with the constraint $\mbf x(t)\in X$.

\noindent\textbf{Notation:} For convenience, we collect all the parameters from the generating equation (Eq.~\eqref{eq:general_pde_}) and introduce the shorthand notation $\mbf G_{\mathrm{p}}$ which represents the labelled tuple of system parameters as
\begin{align}
%&\mbf G_{\mathrm{o}} = \left\{A,~B_{xi},~C_{zx},~D_{zi},~C_{yx},~D_{yi},~C_v,~D_{vw},~D_{vu}\mid i\in \{x,w,u\}\}\right\},\label{eq:ode-params}\\
&\mbf G_{\mathrm{p}} = \left\{A_{0},~A_{1},~A_{2}\right\}.\label{eq:pde-parms}
\end{align}
When this shorthand notation is used to denote a given set of system parameters, it is presumed that all parameters have appropriate dimensions. Now, we define the notion of a solution to the above-described PDE.

\begin{definition}[Definition of solution of a PDE]\label{defn:PDE}
For given $\mbf x^0 \in X$, we say that $\mbf x$ satisfies the PDE defined by $\{n,\mbf G_{\mathrm b}, \mbf G_{\mathrm{p}}\}$ (defined in Eqs.~\eqref{eq:BC-parms} and \eqref{eq:pde-parms}) with initial condition $\mbf{x}^0$ if: a) $\mbf x(t) \in X$ for all $t\ge 0$; b) $\mbf x$ is Frech\'et differentiable with respect to the $L_2$-norm almost everywhere on $\R_+$; c)	$\mbf x(0) = \mbf x^0$; and d) Eq.~\eqref{eq:general_pde_} is satisfied for almost all $t\ge 0$.
\end{definition}

\section{Representing a PDE as a PIE}\label{sec:pde2pie}
In \Cref{sec:PDE}, we presented the PDE form Eq.~\eqref{eq:general_pde_}. Recall that our goal is to find a PIE form of such PDEs so that the computational tools developed for PIEs can be used in analysis/control. Specifically, we will express PDEs introduced in the previous section as PIEs of the form
\begin{align}\label{eq:PIE_}
\mcl{T}\dot{\mbf{x}}_f(t)&=\mcl{A}\mbf x_f(t),\qquad \mbf{x}_f(0) = \hat{\mbf{x}}_f^0\in L_2^{n}.
\end{align}
However, first, we have to verify that such a PIE form exists, which can be verified by testing some additional constraints on $\mbf G_{\mathrm b}$. Specifically, we can verify that for any well-posed PDE of the form given in Section \ref{sec:PDE} with the initial condition $\mbf x^0\in X$, there exists a corresponding PIE with corresponding initial condition $\mbf x_f^0 \in L_{2}^n$ whose solution can be used to construct a solution to the PDE% --- i.e., there exists an invertible map from any initial condition and corresponding solution of the PDE to a corresponding initial condition and solution of the corresponding PIE.
%In the following subsection, we propose a sufficient condition for the existence of such an invertible map.

\subsection{Admissibility of the Boundary Conditions}\label{subsec:assumptions}
The following definition of admissibility imposes a notion of well-posedness on $X$, the domain of the PDE defined by the continuity constraints and the boundary conditions, which, when satisfied, guarantees the existence of a PIE form for the PDE.
\begin{definition}[Admissible Boundary Conditions]
Given the parameters $\{n, \mbf G_{\mathrm b}\}$, we say the pair $\{n,\mbf G_{\mathrm b}\}$ is \textbf{admissible} if $B_T$ is invertible where
\begin{align}\label{as:invertible}
&B_T := B\bmat{T(0)\\T(b-a)}-\int_{a}^{b}B_{I}(s)U_2 T(s-a)ds,
\end{align}
and the functions $T$, and $U_2$ are as defined in \Cref{fig:Gb_definitions}.
\end{definition}
Since $B_T$ is invertible only when it is a square matrix, naturally, we require $n_{BC}=n_S$, i.e., when we have $n_S$ differentiable states, we need $n_S$ boundary conditions for a well-posed solution. Note that the test for invertibility of $B_T$ depends only on the boundary condition parameters and not the dynamics or the initial condition.
Hence, we refer to this test as ``well-posedness of the boundary conditions'', which is not necessarily the same as ``well-posedness of a PDE''.

\subsection{PIE representation of a PDE: $\mcl T$ and $\mcl A$}

Assuming that we have a PDE with admissible $\{n,\mbf G_{\mathrm b}\}$, in this subsection, we propose a PIE form for any PDE of the form Eq.~\eqref{eq:general_pde_} such that the solutions of the PDE and the corresponding PIE are equivalent, in a sense, the existence of the PDE solution guarantees the existence of a PIE solution and vice versa. The conversion of PDE to PIE and the equivalence of the solutions are summarized below.

\begin{restatable}{theorem}{thmEquivalence}\label{thm:equivalence}
Given a set of PDE parameters $\{n,$ $\mbf G_{\mathrm b},$ $\mbf G_{\mathrm{p}}\}$ as defined in \Cref{eq:pde-parms,eq:BC-parms} with $\{n,\mbf G_{\mathrm b}\}$ admissible, let the PI operators $\{\mcl{T},$ $\mcl A\}$ be as defined in \cref{fig:Gb_definitions}. Then, for any $\mbf{x}^0$ $\in$ $X$ ($X$ is as defined in \Cref{eq:odepde_general_domain}), $\{\mbf{x}\}$ satisfies the PDE defined by $\{n,\mbf G_{\mathrm b}, \mbf G_{\mathrm{p}}\}$ with initial condition $\mbf{x}^0$ if and only if $\{\mcl D\mbf{x}\}$ satisfies the PIE defined by $\{\mcl{T},$ $\mcl A\}$ with initial condition $\mcl D\mbf{x}^0 \in L_2^{n_{\mbf x}}$ where $\mcl D\mbf x = \text{col}(\partial_s^0\mbf x_0,\partial_s\mbf x_1,\partial_s^2\mbf x_2)$.
\end{restatable}
\begin{proof}
The proof is simply a matter of using the definitions of $\{\mcl T, \mcl A\}$  and verifying that the definition of
solution is satisfied for both the PDE and PIE. The proof is similar to the proof of Theorem 12 in \cite{shivakumar_SICON2022} with $n=\{n_0,n_1,n_2\}$, $\hat{\mcl T} = \mcl T$, $\hat{\mcl A} = \mcl A$ and other unused parameters set to empty sets or zeros.
\end{proof}

The above result provides a map from the PDE solution, $\mbf x$, to the solution of the corresponding PIE, $\mcl D\mbf x$. However, an inverse map, given by $\mcl T$, exists that maps the solution of the PIE back to the PDE solution. The invertible relation between the solutions can be summarized as follows.

% \begin{theorem}
\begin{restatable}{theorem}{Tmap}\label{thm:T_map}
Given an $n$, and $\mbf G_{\mathrm b}$ with $\{n,\mbf G_{\mathrm b}\}$ admissible, let $\mcl{T}$ be as defined in \cref{fig:Gb_definitions}, $X$ as defined in Eq.~\eqref{eq:odepde_general_domain} and $\mcl D$ $:=$diag$(\partial_s^0 I_{n_0},$ $\partial_s I_{n_1},$ $\partial_s^2 I_{n_2})$. Then we have the following.
\begin{enumerate}[label=(\alph*)]
\item If $\mbf{x} \in X$, then $\mcl D \mbf x \in L_2^{n_{\mbf x}}$ and $\mbf{x} = \mcl{T}\mcl D \mbf x$.
\item For any $\hat{\mbf x} \in L_2^{n_{\mbf x}}$, $\hat{\mcl{T}} \hat{\mbf{x}}\in X$ and $\hat{\mbf{x}} = \mcl D \mcl{T} \hat{\mbf{x}}$.
\end{enumerate}
\end{restatable}
% \end{theorem}

\begin{proof}
This can be verified by substituting the definitions of $\mcl D$ and $\mcl T$. The proof is similar to the proof of Theorem 10 in \cite{shivakumar_SICON2022} with $n=\{n_0,n_1,n_2\}$, $\hat{\mcl T} = \mcl T$ and other unused parameters set to empty sets or zeros.
\end{proof}
This bijective mapping ensures the existence of both solutions when one of them exists. Given the PDE parameters $\{n, \mbf G_{\mathrm b}, \mbf G_{\mathrm p}\}$ we now have a set of formulae to find the corresponding PIE parameters $\{\mcl T,\mcl A\}$. Furthermore, we know that a solution to the PDE equation provides a unique solution to the PIE and vice versa.

% \subsection{PIE representation of a PDE: $\{\mcl T,$ $\mcl A\}$}\label{subsec:pde_equivalence}
% Now that we have found an invertible state transformation of the PDE ($\mbf x \rightarrow \hat{\mbf x}_f={\mcl D} \mbf x$), we construct the corresponding PIE representation by replacing all $\mbf x$ terms in the PDE with $\mbf x=\mcl T (\mcl D\mbf x)=\mcl T \mbf x_f$. This substitution results in a PIE of the form Eq.~\eqref{eq:PIE_} where the relevant PI operators are given in \cref{fig:Gb_definitions}.
% Note that, in the formulae, we have used the notation $\mbf P_c$, which stands for the composition operation of two PI operators.

\section{Equivalence of representations: PDE and PIE}\label{sec:equivalence}
In this section, we show that the solutions to the two representations (PDE and PIE) have equivalent stability properties and present a solvable optimization problem to prove stability. However, we have to define a notion of stability for the two representations.
%%%%%%%%%%%%%%%%%%%%%%%%%%%%%%%%%%%%%%%%%%%%%%%%%%%%%%%%%%%%%%%%%%%%%%%%%%%%%%%%%%%%
%\newpage
\subsection{Definitions of stability}\label{sec:stab_pde}
We define the notion of stability of a PDE based on the stability of $\mbf{x}$ that satisfies the PDE for given initial conditions where the stability is defined with respect to the standard Sobolev norm $H$ which is defined as
\begin{align*}
\norm{\mbf x}_{H} = \sum_{i=0}^2 \norm{\mbf x_i}_{H_i},\quad \norm{\mbf x_i}_{H_i}: = \sum_{j=0}^{i}\norm{\partial_s^j \mbf x_i}_{L_2}.
\end{align*}

\begin{definition}[Exponential Stability of a PDE]
We say a PDE defined by $\{n, \mbf G_{\mathrm b}, \mbf G_{\mathrm p}\}$ is exponentially stable, if there exists constants $M$, $\alpha>0$ such that for any $\mbf{x}^0\in X$, if $\mbf{x}$ satisfies the PDE defined by $\{n, \mbf G_{\mathrm b}, \mbf G_{\mathrm p}\}$ with initial condition $\mbf x^0$ then
\begin{align*}
\norm{\mbf{x}(t)}_{H}\le M\norm{\mbf{x}^0}_He^{-\alpha t} \qquad \text{for all} ~t\ge 0.
\end{align*}
\end{definition}

Similar to the stability of a PDE, we can define the stability of a PIE system based on stability $\mbf x_f$ that satisfies the PIE for some initial conditions and zero inputs. Unlike the stability of PDE, the stability of a PIE is defined with respect to the $L_2$-norm since solutions of PIE need not have spatial continuity.

\begin{definition}[Exponential Stability of a PIE]
We say a PIE defined by $\{\mcl T,\mcl A\}$ is exponentially stable, if there exists constants $M$, $\alpha>0$ such that for any $\mbf{x}_f^0\in L_2^{n}$, if $\mbf{x}_f$ satisfies the PIE defined by $\{\mcl T,\mcl A\}$ with initial condition $\mbf x^0_f$, then $\norm{\mbf{x}_f(t)}_{ L_2}\le M\norm{\mbf x_f^0}_{ L_2}e^{-\alpha t} \qquad \text{for all} ~t\ge 0$.
\end{definition}

Based on these definitions, the goal now is to show that for any $\mbf x$ that satisfies the PDE and $\mbf x_f$ that satisfies the PIE, if $\mbf x = \mcl T\mbf x_f$ with $\mcl T$ invertible, then exponential decay of $\mbf x$ implies exponential decay of $\mbf x_f$ and vice versa. The main hurdle in proving this is the fact that these two notions of stability are defined using different norms ($\norm{\cdot}_H$ and $\norm{\cdot}_{L_2}$). For any $\mbf x$ and $\mbf x_f$ such that $\mbf x = \mcl T\mbf x_f$, we know that $\norm{\mbf x}_{H}\le c$ implies $\norm{\mcl T\mbf x_f}_{L_2}<c$, but the converse is typically not true. Thus, additional steps are needed to prove the converse implication which is accomplished using a new norm on the space $X$, denoted by $\norm{\cdot}_X$, to show that:
\begin{enumerate}
\item $\mcl T$ is a norm-preserving bijection from $L_2$ to $X$ (when equipped with $\norm{\cdot}_X$).
\begin{itemize}
\item[-] Consequently, $X$ is closed under $\norm{\cdot}_X$ (and $\mcl T$ is unitary)
\end{itemize}
\item $\norm{\cdot}_H$ is equivalent to $\norm{\cdot}_X$ on the subspace $X$ (thus stability of PDE w.r.t. $\norm{\cdot}_H$ is equivalent to stability of PIE w.r.t. $\norm{\cdot}_{\mcl T}$)
\item Finally, the PDE $\{n,\mbf G_{\mathrm p},\mbf G_{\mathrm b}\}$ is exponentially stable if and only if the corresponding PIE defined by $\{\mcl T,\mcl A\}$ is exponentially stable (PIE parameters are related to PDE parameters as shown in \cref{fig:Gb_definitions}).
\end{enumerate}

\subsection{The map $\mcl T$ are unitary}\label{subsec:unitary}
First, we would like to show that $X$ is complete with respect to the norm $\norm{\cdot}_{X}$. Previously, in \Cref{thm:T_map}, we showed that $\mcl T$ is invertible. Therefore, we must show that $\mcl T$ preserves the inner product. However, we first define the new $X$-inner product as
\begin{align}\label{eq:ip_xv}
\ip{\mbf{x}}{\mbf{y}}_{X}:= \sum_{i=0}^2 \ip{\partial_s^i\mbf{x}_i}{\partial_s^i\mbf{y}_i}_{L_2}= \ip{\mcl D\mbf x}{\mcl D \mbf y}_{L_2}.
\end{align}

\begin{restatable}{theorem}{thmUnitary}\label{thm:unitary_T}
Suppose $\{n,\mbf G_{\mathrm b}\}$ is admissible, $\mcl T$ is as defined in \Cref{fig:Gb_definitions} and inner product $\ip{\cdot}{\cdot}_X$ is as defined in \Cref{eq:ip_xv}. Then, for any $\mbf{{x}},~ \mbf{{y}}\in L_2^{n_{\mbf x}}$
\begin{align}
&\ip{\mcl T\mbf{x}}{\mcl{T}\mbf{y}}_X = \ip{\mbf{x}}{\mbf{y}}_{L_2}.
\end{align}

\end{restatable}
\begin{proof}
The proof follows directly from the definition of the $X$ inner product and the map $\mcl T$. The proof is similar to the proof of Theorem 18 in \cite{shivakumar_SICON2022} with $n=\{n_0,n_1,n_2\}$, $\hat{\mcl T} = \mcl T$ and other unused parameters set to empty sets or zeros. 
\end{proof}

% \subsection{Equivalence of norms on $\R\times X_v$ and $H$}
Now, we can show that norms induced by the inner products $\ip{\cdot}{\cdot}_{X}$ and $\ip{\cdot}{\cdot}_H$ on $X$ are equivalent and consequently, notions of stability with respect to these norms will be equivalent.

\begin{restatable}{lemma}{lemNormEquivalence}\label{lem:norm_equivalence}
Suppose pair $\{n,\mbf G_{\mathrm b}\}$ is admissible. Then, for any $\mbf{x}\in X$, $\norm{\mbf{x}}_{X}\le \norm{\mbf{x}}_H$ and there exists a constant $c_0>0$ such that $\norm{\mbf{x}}_H\le c_0\norm{\mbf{x}}_{X}$.
\end{restatable}
\begin{proof}
The proof is same as the proof of Lemma 17 in \cite{shivakumar_SICON2022} with $n=\{n_0,n_1,n_2\}$ and other unused parameters set to empty sets or zeros.
\end{proof}

Now that we have established that $X$-norm can be upper bounded by $H$-norm, we can use that to prove the equivalence of the stability of PDE and PIE because $X$-norm on $X$ is isometric to $L_2$-norm on $L_2$.
\begin{restatable}{theorem}{thmStability}\label{thm:stability_equivalence}
Given an $n$ and system parameters $\{\mbf G_{\mathrm b}, \mbf G_{\mathrm{p}}\}$ as defined in \Cref{eq:BC-parms,eq:pde-parms} with $\{n, \mbf G_{\mathrm b}\}$ admissible, suppose $\{\mcl{T}$ $\mcl A\}$ are as defined in \cref{fig:Gb_definitions}. Then, the PDE defined by $\{n,\mbf G_{\mathrm b}, \mbf G_{\mathrm{p}}\}$ is exponentially stable if and only if the PIE defined by $\{\mcl{T}$ $\mcl A\}$ is exponentially stable.
\end{restatable}
\begin{proof}
The proof is a direct application of the stability definitions. The proof is same as the proof of Theorem 22 in \cite{shivakumar_SICON2022} with $n=\{n_0,n_1,n_2\}$ and other unused parameters set to empty sets or zeros.
\end{proof}

Using the above results, we propose the following optimization problem to test the stability of a PDE with integral terms.

\begin{restatable}{theorem}{LPIStability}\label{thm:LPI_stability}
Given a set of PDE parameters $\{n,$ $\mbf G_{\mathrm b},$ $\mbf G_{\mathrm p}\}$, suppose there exist $\alpha, \delta >0$, $\{R_0, R_1, R_2\}$, $\{H_0,H_1,H_2\}$ such that $\threepi{R_i}, \threepi{H_i}\in [\Pi_3]_{n_{\mbf x}, n_{\mbf x}}^+$, $\threepi{R_i}\ge \alpha I$, $\threepi{H_i}\ge \delta \mcl T^*\mcl T$ and $\threepi{H_i} = -\left(\mcl T^*\threepi{R_i}\mcl A + \mcl A^*\threepi{R_i}\mcl T\right)$ where $\{\mcl{T}$, $\mcl A\}$ are as defined in \cref{fig:Gb_definitions}. Then, the PDE defined by $\{n,\mbf G_{\mathrm b}, \mbf G_{\mathrm{p}}\}$ is exponentially stable.
\end{restatable}
\begin{proof}
Suppose $R_i$ and $H_i$ are as stated above. Suppose $\mbf x$ solves the PDE defined by $\{n,$ $\mbf G_{\mathrm b},$ $\mbf G_{\mathrm p}\}$ for some initial condition $\mbf x^0\in X$. Then $\mbf x_f:= \mcl D \mbf x$ solves the PIE defined by $\{\mcl{T}$, $\mcl A\}$ for initial condition $\mcl \mbf x^0$.

Let the Lyapunov function candidate be $V(\mbf x_f) = \ip{\mcl T\mbf x_f}{\threepi{R_i}\mcl T\mbf x_f}_{L_2}$. Then $V(\mbf x_f) \ge \alpha \norm{\mcl T\mbf x_f}_{L_2}^2$ for all $\mbf x_f\in L_2$. Taking the derivative of $V$ with respect to time along the solution trajectories of the PIE, we have
\begin{align*}
&\dot{V}(t) \\
&= \ip{\mcl T\dot{\mbf x}_f(t)}{\threepi{R_i}\mcl T\mbf x_f(t)}_{L_2}+\ip{\mcl T\mbf x_f(t)}{\threepi{R_i}\mcl T\dot{\mbf x}_f(t)}_{L_2}\\
&= \ip{\mcl A\mbf x_f(t)}{\threepi{R_i}\mcl T\mbf x_f(t)}_{L_2}+\ip{\mcl T\mbf x_f(t)}{\threepi{R_i}\mcl A\mbf x_f(t)}_{L_2}\\
&= \ip{\mbf x_f(t)}{\left(\mcl A^*\threepi{R_i}\mcl T+\mcl T^*\threepi{R_i}\mcl T\right)\mbf x_f(t)}_{L_2} \\
&= -\ip{\mbf x_f(t)}{\threepi{H_i}\mbf x_f(t)}_{L_2} \le -\delta\norm{\mcl T\mbf x_f(t)}_{L_2}^2.
\end{align*}

Then, from Gronwall-Bellman inequality,
\[
\norm{\mbf x_f(t)}_{L_2}^2 \le \frac{k}{\alpha}\norm{\mcl D\mbf x^0}_{L_2}^2\exp(-\delta \xi t)
\]
where $k = \norm{\mcl T^*\threepi{R_i}\mcl T}_{\mcl L(L_2)}$ and $\xi = \norm{\mcl T}^2_{\mcl (L_2)}$. Since the initial condition was an arbitrary function, the above inequality is satisfied for any $\mcl D\mbf x^0 \in L_2$, hence the PIE defined by $\{\mcl T, \mcl A\}$ is exponentially stable. Consequently, from \Cref{thm:stability_equivalence}, the PDE is exponentially stable.
\end{proof}

See \cite{peet_2020aut}, for a parametric form of $P_i$ and $H_i$ that allows the use of Linear Matrix Inequalities to enforce positivity constraint. Once the positivity constraint is rewritten as LMI constraints, we can use an SDP solver to find a feasible solution.

\section{Numerical Example}\label{sec:numerical}
This section presents numerical tests for the stability of the two PDEs introduced earlier. The steps involved in setting up the computational problem --- defining the PDE parameters, converting to PIE form, setting up the operator-valued optimization problem, and converting the operator-valued optimization problem to an LMI feasibility test --- are all performed using the PIETOOLS toolbox for MATLAB.
\subsection{Population Dynamics}
Recall the McKendrick PDE model for population dynamics given by
\begin{align*}
\dot{x}(t,s) &= -\partial_s x(t,s) + f(s)x(t,s),\quad s\in[0,1],\\
x(t,0) &= \int_0^1 h(s)x(t,s)ds,
\end{align*}
where $x$ is the density of the population of age $s$ at any given time $t$. The boundary condition can be considered as an approximation of the number of newborns at time $t$.

We will select the kernel in the integral term of the boundary condition as $h(s) = (1-s)s$, which implies that the population outside some normalized limits $[0,1]$ does not contribute to the birth of newborns. We will employ a constant mortality rate $f(s)=c$ and vary the $c\in \R$ to find the mortality rate, $c_0$ below which the population would go extinct (i.e, $\lim_{t\to\infty} x(t,\cdot) = 0$).

By testing the stability of the above PDE for various $c$ values (using the method of bijection), we determined that for mortality rates greater than $-0.740625$, the population goes extinct. That is, the population will survive ($x$ will not converge to zero) when the population growth rate $f(s)>0.740625$.

\subsection{Observer-based control of reaction-diffusion equation}
Consider the reaction-diffusion example presented in the introduction, where the observer is designed based on boundary measurement. We know that the observer gain, $l$, stabilizes the PDE where
\[
l(s) =-\sqrt{\lambda} \frac{I_1\left(\sqrt{\lambda(1-s^2)}\right) }{\sqrt{1-s^2}}.
\]
Since we are interested in the practical implementation, we will approximate the gains $l$ by a polynomial of a fixed order $n$ denoted by $l_n$. Furthermore, the boundary measurements are replaced by an integral. While a typical approach is to replace point measurements by an integral with a Gaussian kernel centered at the point, in this example specifically, we can use the Fundamental Theorem of Calculus to rewrite the point measurement $\partial_s x(t,1)$ exactly as
\[\partial_s x(t,1) = \partial_s x(t,0) +\int_0^1 \partial_s^2 x(t,s) ds = \int_0^1 \partial_s^2 x(t,s) ds.\]
Likewise, we replace the point measurement value $\partial_s \hat{x}(t,1)$ by an integral to get the closed-loop observer PDE as
\begin{align*}
&\dot{x}(t,s) = \lambda x(t,s) + \partial_s^2 x(t,s), \quad s\in [0,1],\\
&\dot{\hat x}(t,s) = \lambda \hat x(t,s) + \partial_s^2 \hat x(t,s) \\
&\qquad+ \int_0^1 l(s) \left(\partial_s^2 x(t,\theta) - \partial_s^2 \hat x(t,\theta)\right)d\theta,\\
&x(t,0) = 0, \quad x(t,1) = 0,\\
&\hat x(t,0) = 0, \quad \hat x(t,1) = 0.
\end{align*}
Then, using the conversion formulae presented in the appendix, we can find the PIE representation for the closed-loop PDE where $l$ replaced by $l_n$ which is the $n^{th}$ order polynomial approximation of $l$.

For $\lambda\le 5$, we can prove that $1^{st}$-order polynomial approximation (a straight line) is sufficient to guarantee the stability of the closed-loop PDE system. However, as $\lambda$ becomes large higher order polynomial approximation of $l$ ($l_4$ for $\lambda =6$ and so on) was necessary to prove the stability.

\section{Conclusion}\label{sec:conclusion}
To conclude, we presented a standard parametric form for linear PDEs in one spatial dimension with spatial integral terms. We provided a sufficient criterion that guarantees the existence of a PIE representation for such PDEs. We also presented results showing the equivalence of solutions and stability properties of the two representations. Finally, using the equivalence in the stability properties, we formulated the test for stability of PDEs as an optimization problem that can be solved using SDP solvers and demonstrated the application using numerical examples.

\section*{Acknowledgement}
This work was supported by the National Science Foundation under grants No. 1739990 and 1935453

\balance

\bibliography{references}

\newpage

\begin{figure}[!h]
\centering
\scalebox{1}{
\begin{minipage}{\textwidth}
\section*{APPENDIX}
\begin{align*}
& n_{\mbf x}=\sum_{i=0}^2 n_i,\quad n_{Si}=\sum_{j=0}^i n_{j},\quad n_{S}=\sum_{i=0}^2 n_{Si},\quad T(s) = \bmat{I_{n_1}&(s-a)I_{n_2}\\0&I_{n_2}} \in \R^{n_S \times n_S},\\
&U_{2i} = \bmat{0_{n_i\times n_{i+1:2}}\\ I_{n_{i+1:2}}} \in \R^{n_{S_{i}} \times n_{S_{i+1}}},\qquad U_2 = \bmat{\text{diag}(U_{21},U_{22})\\0_{n_2\times n_S}}\in \R^{ \left(n_{\mbf x}+n_S\right)\times n_S},\\
&B_T := B\bmat{T(0)\\T(b-a)}-\int_{a}^{b}B_{I}(s)U_2 T(s-a)ds\in \R^{n_{BC}\times n_S},\quad U_{1i} = \bmat{I_{n_i} \\ 0_{n_{i+1:2}\times n_i}},\\
&U_1 = \bmat{U_{10}&&\\&U_{11}&\\&&U_{12}}\in \R^{(n_S+n_{\mbf x})\times n_{\mbf x}},\quad Q(s)= \bmat{0&I_{n_1}&0\\0&0&(b-s)I_{n_2}\\0&0&I_{n_2}}\in \R^{n_S \times n_\mbf x},\\
&B_Q(s) = B_T^{-1}\left(B_I(s)U_1-B\bmat{0\\Q(b-s)}+\int_s^b B_I(\theta)U_2Q(\theta-s)d\theta\right),\\
&G_0 = \bmat{I_{n_0}&\\&0_{(n_{\mbf x}-n_0)}},\qquad G_2(s,\theta) = \bmat{0\\T_1(s-a)B_Q(\theta)},\qquad G_1(s,\theta) = \bmat{0\\Q_1(s-\theta)}+G_2(s,\theta),\qquad\\
&R_{D,2}(s,\theta) = U_2T(s-a)B_Q(\theta), \qquad R_{D,1}(s,\theta) = R_{D,2}(s,\theta)+U_2Q(s-\theta), \qquad \hat{A}_0(s) = A_0(s)U_1,\\
&\hat{A}_1(s,\theta) = A_0(s)R_{D,1}(s,\theta)+A_1(s,\theta)U_1+\int_a^{\theta} A_1(s,\beta)R_{D,2}(\beta,\theta)d\beta+\int_{\theta}^s A_1(s,\beta)R_{D,1}(\beta,\theta)d\beta \\&\quad + \int_{s}^b A_2(s,\beta)R_{D,1}(\beta,\theta)d\beta,\\
&\hat{A}_2(s,\theta) = A_0(s)R_{D,2}(s,\theta)+A_2(s,\theta)U_1+\int_a^{s} A_1(s,\beta)R_{D,2}(\beta,\theta)d\beta+\int_s^{\theta} A_2(s,\beta)R_{D,2}(\beta,\theta)d\beta \\&\quad + \int_{\theta}^b A_2(s,\beta)R_{D,1}(\beta,\theta)d\beta,\\
& \mcl{T} = \threepi{G_i},\qquad \mcl A = \threepi{\hat A_i}.
\end{align*}
\caption{Definitions based on $\{n,A_i,B,B_I\}$.}\label[EqnBlock]{fig:Gb_definitions}
\end{minipage}}

\end{figure}

\end{document}